\documentclass[11pt]{article}

\usepackage[T1]{fontenc}
\usepackage{lmodern}
\usepackage[a4paper,margin=0.9in]{geometry}
\usepackage{microtype}
\usepackage{amsmath,amssymb,amsthm,mathtools,bm}
\usepackage{booktabs}
\usepackage{graphicx}

\usepackage{xcolor}
\usepackage{caption}

\DeclareCaptionLabelSeparator{revisionperiod}{.\space}
\usepackage{float}
\usepackage{enumitem}
\usepackage[numbers,sort&compress]{natbib}
\usepackage[hidelinks]{hyperref}

\numberwithin{equation}{section}
\allowdisplaybreaks
\setlist{nosep}

\newtheorem{theorem}{Theorem}[section]
\newtheorem{proposition}[theorem]{Proposition}
\newtheorem{lemma}[theorem]{Lemma}
\newtheorem{corollary}[theorem]{Corollary}
\newtheorem{definition}[theorem]{Definition}
\theoremstyle{remark}

\newcommand{\R}{\mathbb R}

\newcommand{\E}{\mathbb E}
\newcommand{\Var}{\operatorname{Var}}
\newcommand{\SCV}{\operatorname{SCV}}
\newcommand{\e}{\mathrm e}
\newcommand{\ii}{\mathrm i}
\newcommand{\dd}{\,\mathrm d}
\newcommand{\Tcal}{\mathcal T}
\newcommand{\Hcal}{\mathcal H}
\newcommand{\abs}[1]{\left|#1\right|}

\newcommand{\vH}{v^{\mathrm H}}
\newcommand{\rhoH}{\rho^{\mathrm H}}
\DeclareMathOperator*{\argmin}{arg\,min}

\title{\textbf{Least-Variable Harmonic Matrix-Exponential Distributions}}
\author{Maria Laura Battagliola and Oscar Peralta}
\date{}

\begin{document}
\maketitle

\begin{abstract}
Concentrated matrix-exponential (CME) distributions are random clocks approximating a fixed time, with quality measured by the squared coefficient of variation (SCV); a well-known construction builds such clocks from products of cosine-squared terms under shared exponential damping, but the number of these terms, and hence the number of parameters to optimize, grows with the order. Since exhaustive search over the full parameter space becomes prohibitively expensive at high orders, previous work has relied on low-dimensional heuristic parametrizations rather than the true optimum. We introduce a strictly larger class of common-damping harmonic densities, built from an arbitrary nonnegative trigonometric polynomial rather than such a product. For each fixed pair of scalar parameters, coefficient optimization reduces to a single eigenvalue problem, leaving a two-dimensional nonlinear search independent of the order. The enlargement leaves the infimum unchanged from that of the classical cosine-squared construction. Applying this exact characterization at the orders where the three-parameter heuristic allows comparison gives smaller reported SCV values, with larger gains at higher orders. Separately, an explicit harmonic construction gives an $O(N^{-2})$ upper bound on the attainable SCV, where $N$ is the ME representation budget, compared with the Erlang distribution's linear rate $1/N$.
\end{abstract}

\paragraph{Keywords.}
Concentrated matrix-exponential distributions; squared coefficient of variation; harmonic kernels; Toeplitz matrices; randomization.

\paragraph{Mathematics Subject Classification (2020).}
Primary 60E05. Secondary 42A05, 15B05.

\section{Introduction}\label{sec:introduction}

Randomization methods replace a deterministic time $t_0$ by a random horizon $t_0U$, where $U$ is a nonnegative random variable with mean one. The purpose is to retain concentration around $t_0$ while choosing the law of $U$ so that the randomized problem is easier to evaluate. Erlangization is the classical example, in which a mean-one Erlang random variable of order $N$ has variance $1/N$, and this is the smallest variance among mean-one phase-type distributions of order $N$ \citep{AldousShepp1987}.

A matrix-exponential (ME) distribution of order $N$ has density
\begin{equation}\label{eq:me-density}
  f(t)=\bm\alpha\,\e^{\bm A t}(-\bm A)\bm 1,
  \qquad t\geq0,
\end{equation}
for a row vector $\bm\alpha\in\R^N$, a matrix $\bm A \in \R^{N \times N}$, and the column vector of ones $\bm 1\in\R^N$, subject to the requirement that $f$ be a probability density. A phase-type distribution is the special case of \eqref{eq:me-density} in which $(\bm\alpha,\bm A)$ describes a finite-state Markov jump process with absorption, so that $\bm\alpha$ is nonnegative and $\bm A$ is a transient intensity matrix. The order $N$ is simply the size of this matrix representation. Equivalently, it is the dimension of the real matrix used to represent the finite collection of exponential and oscillatory modes of the density. ME distributions retain the finite-dimensional matrix representation and rational Laplace transform of phase-type distributions, but do not require the parameters themselves to have this Markovian sign structure. See \citet[Chapters~3--4]{BladtNielsen2017} for a comprehensive treatment of phase-type and ME distributions.

The concentrated matrix-exponential (CME) literature exploits this freedom to choose $\bm\alpha$ and $\bm A$ without a Markovian sign structure to construct positive random clocks with much smaller variability. The original CME construction, introduced by \citet{HorvathSafarTelekZambo2016}, is the common-damping harmonic family generated by the density function
\begin{equation}\label{eq:intro-original-family}
  f_{\omega,\bm\phi}(t)
  ={c\,}\e^{-t}\prod_{j=1}^{L}
    \cos^2\!\left(\frac{\omega t-\phi_j}{2}\right),
  \qquad t\geq0,
\end{equation}
where $\bm\phi=(\phi_1,\ldots,\phi_L)$. Here, $L$ is the number of cosine factors, $\omega>0$ is their common angular frequency, $\phi_1,\ldots,\phi_L\in\R$ are the corresponding angular offsets, and $c>0$ is a normalizing constant chosen so that $f_{\omega,\bm\phi}$ integrates to one.

Applications of this family include non-overshooting numerical inverse Laplace transformation \citep{HorvathTalyigasTelek2018,HorvathEtAl2020NILT}, numerical inverse $Z$-transformation \citep{HorvathMeszarosTelek2020}, and deterministic-time approximation through ME-fication \citep{AkarEtAl2021MEfication,Telek2022Transient}, with \citet{AlmousaEtAl2022CME} giving the computational account of the CME method. High-order numerical optimization has also been pursued directly \citep{HorvathHorvathTelek2020,AlmousaTelek2021}, alongside a related construction with real eigenvalues, the CME-R family \citep{MeszarosTelek2022}. For large $L$, searching for the least variable member of this family produces a high-dimensional nonconvex optimization problem, which an analytic predecessor avoids by giving an explicit common-damping CME family without numerical optimization \citep{BattagliolaPeralta2026OptFree}.

Our contributions in this paper are twofold. First, we generalize the cosine-squared construction \eqref{eq:intro-original-family} to a strictly larger class of common-damping harmonic densities, which we call the \emph{harmonic CME family} and define formally in Definition~\ref{def:harmonic-cme}. For each fixed pair of scalar parameters, coefficient optimization over this larger class reduces to a single Hermitian generalized eigenvalue problem, and the resulting infimum equals that of the classical cosine-squared construction. At the orders where a three-parameter heuristic from the literature allows comparison, the numerical values obtained from this characterization are smaller, with larger gains at higher orders, where an exhaustive numerical search over the original $(L+1)$-dimensional parametrization becomes prohibitively expensive.

Second, our contribution is asymptotic. Writing $\vH_L$ for the SCV infimum within the harmonic CME family of degree $L$, we construct an explicit sequence achieving
\begin{equation}\label{eq:intro-limsup}
  \limsup_{L\to\infty}(2L+1)^2\vH_L\leq1.
\end{equation}
This gives an $O(N^{-2})$ upper bound on the attainable SCV, where $N=2L+1$ is the representation budget, compared with the linear rate $1/N$ for a mean-one Erlang clock. The construction separates the clock's within-period location from an independent geometric number of complete periods, controlling the former with a sine-weighted Fourier polynomial and a cosine-power tail suppressor and the latter with logarithmically growing damping.

The paper is organized as follows. Section~\ref{sec:harmonic-clocks} introduces the rescaled harmonic CME class and the quadratic scalarization. Section~\ref{sec:reduction} gives the Toeplitz representation and the exact two-parameter reduction. Section~\ref{sec:asymptotics} proves the asymptotic upper bound \eqref{eq:intro-limsup}. Section~\ref{sec:numerics} gives a short numerical comparison with previous optimization attempts. Section~\ref{sec:discussion} discusses the relation with Erlang randomization and the polynomial-square repeated-pole problem.

\section{Harmonic random clocks}\label{sec:harmonic-clocks}

This section defines the harmonic CME family and studies basic properties of the squared coefficient of variation. Moreover, it expresses a monotone transform of the SCV as the minimum of a relative quadratic loss over a scalar parameter, which supplies the second of the two outer variables used throughout the paper.

\subsection{The rescaled harmonic CME family}

For a nonzero nonnegative kernel $f:[0,\infty) \to [0,\infty)$ with finite positive moments $m_0(f),m_1(f),m_2(f)$, of orders zero, one and two, respectively, write
\begin{equation}\label{eq:moments-scv}
  m_r(f)=\int_0^\infty z^r f(z)\dd z,
  \qquad
  \SCV(f)=\frac{m_0(f)m_2(f)}{m_1(f)^2}-1.
\end{equation}
The SCV in \eqref{eq:moments-scv} is unchanged by multiplication of $f$ by a positive constant or by rescaling of the underlying random variable, so we can work with unnormalized rescaled kernels throughout.

Because of this invariance, set
\[
  z=\omega t,
  \qquad
  \beta=\frac1\omega.
\]
Up to a positive normalizing factor, the classical density \eqref{eq:intro-original-family} then becomes
\begin{equation}\label{eq:intro-scaled-family}
  \e^{-\beta z}\prod_{j=1}^{L}
  \cos^2\!\left(\frac{z-\phi_j}{2}\right),
  \qquad z\geq0.
\end{equation}
The product of the squared-cosine factors is a nonnegative trigonometric polynomial of degree at most $L$: explicitly, each factor is itself a degree-one trigonometric polynomial,
\begin{equation}\label{eq:cosine-factor-expansion}
  \cos^2\!\left(\frac{z-\phi_j}{2}\right)
  =\frac12+\frac14\e^{\ii(z-\phi_j)}+\frac14\e^{-\ii(z-\phi_j)},
\end{equation}
so multiplying the $L$ factors together and collecting powers of $\e^{\ii z}$ expands the product into a sum $T(z)=\sum_{k=-L}^{L}c_k\e^{\ii kz}$, with the top frequencies $k=\pm L$ arising from picking the $\e^{\pm\ii(z-\phi_j)}$ term in \eqref{eq:cosine-factor-expansion} for every factor $j$. We therefore consider the larger class obtained by allowing an arbitrary nonzero polynomial in the cone of such sums.

\begin{definition}[Harmonic CME family]\label{def:harmonic-cme}
Let $\Tcal_L^+$ denote the cone of nonnegative real trigonometric polynomials of degree at most $L$,
\[
  \Tcal_L^+
  =\left\{
    T(z)=\sum_{k=-L}^{L}c_k\e^{\ii kz}
    \ \middle|\
    c_{-k}=\overline{c_k},\ T(z)\geq0\ \text{for all }{z \geq 0}
   \right\}.
\]
For $\beta>0$, the harmonic CME family of degree $L$ and damping $\beta$ is
\[
  \Hcal_L(\beta)
  =\left\{z\mapsto\e^{-\beta z}T(z)
    \ \middle|\
    T\in\Tcal_L^+,\ T\not\equiv0\right\}.
\]
\end{definition}

The classical common-damping harmonic family in \eqref{eq:intro-original-family} is therefore the particular case of $\Hcal_L(\beta)$ in which $T$ is a product of $L$ squared-cosine factors, rather than a general element of $\Tcal_L^+$. The advantage of the rescaled coordinate is that the harmonic period is now always $2\pi$, with the single parameter $\beta$ controlling the exponential decay from one period to the next.

Let $f(z)=\e^{-\beta z}T(z)$ with $T(z)=\sum_{k=-L}^{L}c_k\e^{\ii kz}$. Its Laplace transform is
\[
  \mathcal L f(s)
  =\int_0^\infty \e^{-sz}f(z)\dd z
  =\sum_{k=-L}^{L}\frac{c_k}{s+\beta-\ii k},
  \qquad \Re(s)>-\beta.
\]
Equivalently,
\[
  \mathcal L f(s)
  =\frac{q(s)}{\displaystyle\prod_{k=-L}^{L}(s+\beta-\ii k)},
  \qquad \deg q\leq 2L,
\]
where $q$ is a polynomial of degree at most $2L$. After cancellation of common factors, the actual pole set, and hence the minimal representation order, may be smaller. Hence, the possible poles are
\[
  -\beta+\ii k,
  \qquad {k=-L,\ldots,L}.
\]
The pole at $-\beta$ contributes one real mode, while each conjugate pair $-\beta\pm\ii k$, $k=1,\ldots,L$, contributes two real dimensions to the ME order. Thus, the resulting ME representation has order at most $2L+1$. We call this quantity the representation budget $N=2L+1$.

Define the infimum over all damping parameters and admissible envelopes, not just the cosine-squared special case above, by
\[
  \vH_L
  =\inf_{\beta>0}\ \inf_{f\in\Hcal_L(\beta)}\SCV(f).
\]

\subsection{Relative quadratic loss}

The proposition below expresses a monotone transform of the SCV in \eqref{eq:moments-scv} as the minimum, over one positive scalar, of a relative quadratic loss. This auxiliary scalar becomes, together with $\beta$, one of the two scalar variables left in the outer nonlinear optimization once the coefficient vector has been eliminated.

\begin{proposition}\label{prop:relative-loss}
Let $f\geq0$ have finite positive moments $m_0(f),m_1(f),m_2(f)$, and for $\theta>0$ define
\begin{equation}\label{eq:relative-loss}
  R_\theta(f)
  =\frac{\int_0^\infty(z-\theta)^2f(z)\dd z}
         {\theta^2\int_0^\infty f(z)\dd z}.
\end{equation}
Then
\begin{equation}\label{eq:theta-star}
 { \argmin_{\theta>0}R_\theta(f)=\frac{m_2(f)}{m_1(f)}}
\end{equation}
and, writing $\rho(f)$ for the resulting minimum,
\begin{equation}\label{eq:rho-transform}
{
  \rho(f)=\min_{\theta>0}R_\theta(f)
  =1-\frac{m_1(f)^2}{m_0(f)m_2(f)}
  =\frac{\SCV(f)}{1+\SCV(f)}.}
\end{equation}
\end{proposition}

\begin{proof}
With $y=1/\theta$,
\[
  R_{1/y}(f)
  =1-2\frac{m_1(f)}{m_0(f)}y+\frac{m_2(f)}{m_0(f)}y^2.
\]
This is a strictly convex quadratic in $y>0$, minimized at $y=m_1(f)/m_2(f)$. Substitution gives \eqref{eq:rho-transform}.

\end{proof}

Since the inverse map of \eqref{eq:rho-transform} is increasing on $[0,1)$, minimizing the $\SCV(f)$ is equivalent to minimizing $\rho(f)$. For fixed $\theta$, both the numerator and the denominator in \eqref{eq:relative-loss} are linear functionals of the nonnegative envelope $T$. This fact will also be used to recover the cosine-square form \eqref{eq:intro-scaled-family} from the larger trigonometric cone.

\section{Two-parameter reduction}\label{sec:reduction}

This section contains the main finite-order result of the paper. Specifically, it shows that the coefficient optimization for fixed $(\beta,\theta)$ is an inner generalized eigenvalue problem, while $(\beta,\theta)$ themselves are the two remaining variables to be optimized numerically. The resulting two-parameter formulation has the same optimal value as the original cosine-square family.

\subsection{Spectral factors and Toeplitz moments}

The classical Fej\'er--Riesz factorization \citep[Section~1]{DritschelRovnyak2010} states that every $T\in\Tcal_L^+$ can be written as
\begin{equation}\label{eq:fejer-riesz}
  T(z)=\abs{P(\e^{\ii z})}^2,
\end{equation}
where 
\[
     P(w)=\sum_{k=0}^{L}p_kw^k.
\]
Conversely, \eqref{eq:fejer-riesz} is a nonnegative trigonometric polynomial of degree at most $L$. Thus, the harmonic cone can be parameterized by the coefficient vector $\bm p=(p_0,\ldots,p_L)^{\top}$, with no positivity constraints left on $\bm p$, where $\top$ denotes ordinary transpose. Since the coefficients may be complex, we write the conjugate transpose as $\overline{\bm p}^{\top}$.

For $r=0,1,2$ and $\beta>0$, define the $(L+1)\times(L+1)$ Hermitian Toeplitz matrix $M_r(\beta)$ by its entries
\begin{equation}\label{eq:moment-matrix}
  [M_r(\beta)]_{jk}
  =\frac{r!}{\{\beta-\ii(k-j)\}^{r+1}},
  \qquad 0\leq j,k\leq L.
\end{equation}

\begin{proposition}\label{prop:moment-identity}
Let
\[
  f_{\bm p,\beta}(z)
  =\e^{-\beta z}\abs{P(\e^{\ii z})}^2.
\]
Then, for $r=0,1,2$,
\begin{equation}\label{eq:moment-quadratic}
  m_r(f_{\bm p,\beta})=\overline{\bm p}^{\top}M_r(\beta)\bm p.
\end{equation}
Moreover, $M_0(\beta)$ is positive definite for every $\beta>0$.
\end{proposition}

\begin{proof}
Since
\[
  \abs{P(\e^{\ii z})}^2
  =\sum_{j,k=0}^{L}\overline{p_j}p_k\e^{\ii(k-j)z},
\]
the moment of order $r$ is obtained by integrating the terms
\[
  \int_0^\infty z^r\e^{-\beta z}\e^{\ii(k-j)z}\dd z
  =\frac{r!}{\{\beta-\ii(k-j)\}^{r+1}},
\]
which proves \eqref{eq:moment-quadratic}. If $\bm p\neq0$, then
\[
  \overline{\bm p}^{\top}M_0(\beta)\bm p
  =\int_0^\infty\e^{-\beta z}\abs{P(\e^{\ii z})}^2\dd z>0,
\]
since $P$ is a nonzero polynomial, and therefore has only finitely many zeros. In particular, it cannot vanish on the entire unit circle.

\end{proof}

\subsection{The generalized eigenvalue problem}

For $\beta,\theta>0$, Proposition~\ref{prop:moment-identity} writes the numerator of the relative quadratic loss as a quadratic form in $\bm p$ with matrix
\[
  A_L(\beta,\theta)
  =M_2(\beta)-2\theta M_1(\beta)+\theta^2M_0(\beta).
\]
Define
\begin{equation}\label{eq:lambda-profile}
  \lambda_L(\beta,\theta)
  =\min_{\bm p\neq0}
    \frac{\overline{\bm p}^{\top}A_L(\beta,\theta)\bm p}
         {\theta^2\overline{\bm p}^{\top}M_0(\beta)\bm p}.
\end{equation}
The ratio minimized in \eqref{eq:lambda-profile} is a generalized Rayleigh quotient, since $A_L(\beta,\theta)$ is Hermitian, being a linear combination of Hermitian matrices with real coefficients, and $\theta^2 M_0(\beta)$ is Hermitian positive definite as shown in Proposition~\ref{prop:moment-identity}. It follows that $\lambda_L(\beta,\theta)$ is the smallest generalized eigenvalue of
\[
  A_L(\beta,\theta)\bm p
  =\lambda\theta^2M_0(\beta)\bm p.
\]
Equivalently, after a positive-definite square-root or Cholesky reduction of $\theta^2M_0(\beta)$, this becomes an ordinary Hermitian Rayleigh-quotient problem. We use $\lambda_L(\beta,\theta)$ for this smallest generalized eigenvalue throughout. See, for example, \citet{GolubVanLoan2013} for the standard theory of Hermitian definite generalized eigenvalue problems.

\begin{theorem}[Two-parameter characterization]\label{thm:two-parameter}
Define
\[
  \rhoH_L
  =\inf_{\beta>0,\,\theta>0}\lambda_L(\beta,\theta).
\]
Then
\begin{equation}\label{eq:v-from-rho}
  \vH_L=\frac{\rhoH_L}{1-\rhoH_L}.
\end{equation}
Thus the full coefficient optimization is absorbed into $\lambda_L(\beta,\theta)$, and the remaining nonlinear optimization has only the two scalar variables $(\beta,\theta)$, independently of $L$.
\end{theorem}

\begin{proof}
By Proposition~\ref{prop:moment-identity}, substituting the three moment identities into \eqref{eq:relative-loss} shows that the generalized Rayleigh quotient in \eqref{eq:lambda-profile} is exactly $R_\theta(f_{\bm p,\beta})$. Fej\'er--Riesz factorization shows that varying $\bm p$ ranges over all nonnegative trigonometric envelopes in $\Tcal_L^+$. Therefore
\[
  \inf_{\beta>0,\theta>0}\lambda_L(\beta,\theta)
  =\inf_{\beta>0}\inf_{f\in\Hcal_L(\beta)}\inf_{\theta>0}R_\theta(f).
\]
Applying Proposition~\ref{prop:relative-loss} and then the monotone transformation in \eqref{eq:rho-transform} yields \eqref{eq:v-from-rho}.
\end{proof}

For any fixed pair $(\beta,\theta)$, a minimizing coefficient vector is recovered as any nonzero generalized eigenvector associated with $\lambda_L(\beta,\theta)$. Its entries give $P(w)=\sum_{k=0}^Lp_kw^k$ and hence $T(z)=|P(\e^{\ii z})|^2$. The eigenvector determines $P$ up to a nonzero complex scalar and $T$ up to a positive scalar. In particular, if the outer minimum is attained, the full harmonic kernel is recovered directly from the minimizing pair $(\beta,\theta)$.

For a fixed pair $(\beta,\theta)$, let $f_{\bm p,\beta}$ be recovered from a minimizing eigenvector. Proposition~\ref{prop:relative-loss} gives the exact inequality
\[
  \SCV(f_{\bm p,\beta})
  \leq\frac{\lambda_L(\beta,\theta)}{1-\lambda_L(\beta,\theta)},
  \qquad \lambda_L(\beta,\theta)<1,
\]
with equality precisely when $\theta=m_2(f_{\bm p,\beta})/m_1(f_{\bm p,\beta})$, as shown in \eqref{eq:theta-star}. This holds for the recovered $f_{\bm p,\beta}$ whenever $\lambda_L(\beta,\theta)<1$, whether or not the outer search over these two variables has converged, so a valid upper bound on $\vH_L$ exists for any such recovered kernel. Its numerical evaluation is addressed in Section~\ref{sec:numerics}.

\subsection{Return to cosine-square products}

For fixed $(\beta,\theta)$, normalizing the envelope makes the relative loss a linear functional on a compact convex slice of $\Tcal_L^+$, so a minimizing envelope can be chosen on an extreme ray of the cone. A ray $\{aT:a\geq0\}$ is called extreme if every decomposition $T=T_1+T_2$ with $T_1,T_2\in\Tcal_L^+$ forces both summands to be nonnegative scalar multiples of $T$, making an extreme ray a one-dimensional edge of the cone. The extreme rays characterized below take the form of products of squared sines, which is only a reparametrization of the original family, since
\[
\abs{\e^{\ii z}-\e^{\ii\phi_j}}^2
=4\sin^2\left(\frac{z-\phi_j}{2}\right)
=4\cos^2\left(\frac{z-(\phi_j+\pi)}{2}\right).
\]

\begin{proposition}\label{prop:extreme-rays}
A nonzero $T\in\Tcal_L^+$ spans an extreme ray of $\Tcal_L^+$ if and only if it has total zero multiplicity $2L$ on the unit circle. Equivalently, up to a positive constant,
\begin{equation}\label{eq:extreme-product}
  T(z)
  =\prod_{j=1}^{L}\abs{\e^{\ii z}-\e^{\ii\phi_j}}^2
  =4^L\prod_{j=1}^{L}
    \sin^2\!\left(\frac{z-\phi_j}{2}\right),
\end{equation}
where repeated offsets are allowed.
\end{proposition}

\begin{proof}
Every zero of a nonzero nonnegative real trigonometric polynomial has finite even multiplicity. Indeed, if the first nonzero term in the local expansion at a zero had odd order, it would change sign across that zero, contradicting nonnegativity. Let the distinct zeros of $T$ have multiplicities $2m_1,\ldots,2m_s$ and put $d=m_1+\cdots+m_s\leq L$. The factor $4\sin^2((z-\phi)/2)$ has zero multiplicity two but trigonometric degree one, so $d$ counts the amount of trigonometric degree forced by these zeros. Then
\[
  T(z)=Z(z)G(z),
  \qquad
  Z(z)=\prod_{r=1}^{s}
  \left[4\sin^2\!\left(\frac{z-\phi_r}{2}\right)\right]^{m_r},
\]
where $G$ is strictly positive and has degree at most $L-d$. This factorization removes exactly the prescribed unit-circle zeros. Equivalently, it follows by factoring the corresponding roots from a Fej\'er--Riesz spectral factor.

If $d<L$, choose a nonzero real trigonometric polynomial $S$ of degree at most $L-d$ which is not proportional to $G$. Because $G$ is continuous, periodic and strictly positive, it is bounded away from zero, while $S$ is bounded. Hence, for sufficiently small $\varepsilon>0$, both $G+\varepsilon S$ and $G-\varepsilon S$ remain nonnegative. Thus
\[
  T=\frac12 Z(G+\varepsilon S)+\frac12 Z(G-\varepsilon S),
\]
is a nontrivial decomposition in the cone, so the ray is not extreme.

If $d=L$ and $T=T_1+T_2$ with $T_1,T_2\in\Tcal_L^+$, then both summands must vanish at every zero of $T$ to at least the same even multiplicity. Otherwise a lower-order nonnegative term from one summand could not be cancelled by the other, and $T$ itself would have the lower multiplicity. Hence, both are divisible by the degree-$L$ factor $Z$. The degree bound leaves only constant quotients, so each $T_i$ is a nonnegative scalar multiple of $Z$, proving extremality and \eqref{eq:extreme-product}.
\end{proof}

\begin{corollary}
\label{cor:cosine-equivalence}
For every $L\geq0$,
\[
  \vH_L
  =\inf_{\beta>0,\,\phi_1,\ldots,\phi_L\in\R}
  \SCV\!\left(
    \e^{-\beta z}\prod_{j=1}^{L}
    \cos^2\!\left(\frac{z-\phi_j}{2}\right)
  \right).
\]
\end{corollary}

\begin{proof}
Fix $(\beta,\theta)$. Because $R_\theta$ is unchanged when $T$ is multiplied by a positive constant, every nonzero ray of $\Tcal_L^+$ has a unique representative under the normalization
\[
  \int_0^\infty \e^{-\beta z}T(z)\dd z=1.
\]
On this normalized slice, the denominator of $R_\theta$ is the fixed constant $\theta^2$, while the numerator is linear in $T$. Thus, $R_\theta$ becomes a continuous linear functional of $T$. The normalized slice is compact and convex, so a minimum can be taken at an extreme point. Since each ray meets the slice once, its extreme points correspond exactly to the extreme rays of the original cone. By Proposition~\ref{prop:extreme-rays}, such an envelope is a product of $L$ squared sines. Since $\sin\vartheta=\cos(\vartheta-\pi/2)$, each factor satisfies $\sin^2((z-\phi_j)/2)=\cos^2((z-\phi_j-\pi)/2)$, so shifting every offset by $\pi$ gives a product of squared cosines with the same zero locations. Finally take the infimum over $(\beta,\theta)$ and use Proposition~\ref{prop:relative-loss}.
\end{proof}

By Corollary~\ref{cor:cosine-equivalence}, the common infimum can be computed through the spectral factor and the two-variable profile \eqref{eq:lambda-profile}.

\section{Large-order behavior}\label{sec:asymptotics}
Recall that the rescaled formulation in \eqref{eq:intro-scaled-family} separates two sources of variability. A harmonic envelope can concentrate most of its mass inside one period, but periodicity repeats the same profile in every later period. The exponential factor then determines how much mass is assigned to those replicas. The aim of this section is constructive. Rather than identify the optimizing $\beta$ and envelope at each finite order, we exhibit one admissible sequence whose SCV is $O(N^{-2})$. The construction must balance two effects. Increasing $\beta$ suppresses the replicated periods, but the same exponential tilt favors mass near the beginning of the first period and can strongly amplify a small part of a peak that wraps across $2\pi$. We first separate the within-period and between-period contributions exactly, then construct a sharply concentrated first-period profile for which this trade-off can be controlled.

\subsection{Decomposition over harmonic periods}

This decomposition is exact for every nonnegative $2\pi$-periodic envelope, and identifies the damping parameter $\beta$ with the geometric weight assigned to successive copies of one period.

The same geometric cell decomposition was used in \citet[Lemma~3 and equation~(12)]{BattagliolaPeralta2026OptFree} for the powered-Fej\'er CME family. Importantly, the argument depends only on periodicity and exponential damping, so it applies to every admissible harmonic envelope, not only to that explicit construction.

Let $T$ be a nonzero nonnegative $2\pi$-periodic function and let $Z$ have density proportional to
\[
  \e^{-\beta z}T(z),\qquad z\geq0,
\]
for some $\beta>0$. Define $\Theta$ on $[0,2\pi)$ by the density proportional to
\[
  \e^{-\beta\theta}T(\theta),
  \qquad 0\leq\theta<2\pi,
\]
and put
\[
  q=\e^{-2\pi \beta}.
\]

\begin{proposition}\label{prop:period-decomposition}
Let $K$ be independent of $\Theta$ with geometric distribution
\[
  \Pr(K=k)=(1-q)q^k,
  \qquad k=0,1,2,\ldots.
\]
Then
\begin{equation}\label{eq:period-sum}
  Z\stackrel d=\Theta+2\pi K.
\end{equation}
Consequently,
\begin{align}
  \E[Z]
  &=\E[\Theta]+\frac{2\pi q}{1-q},
  \label{eq:period-mean}\\
  \Var(Z)
  &=\Var(\Theta)+\frac{(2\pi)^2q}{(1-q)^2}.
  \label{eq:period-variance}
\end{align}
\end{proposition}

\begin{proof}
On the $k$th period write $z=\theta+2\pi k$ with $0\leq\theta<2\pi$. Periodicity gives $T(z)=T(\theta)$ and the exponential factor contributes
\[
  \e^{-\beta(\theta+2\pi k)}
  =q^k\e^{-\beta\theta}.
\]
After normalization, the period index is geometric and its conditional position inside the period has the same law for every $k$. This gives \eqref{eq:period-sum}, while \eqref{eq:period-mean} and \eqref{eq:period-variance} follow from independence.
\end{proof}

Since $\vH_L$ is an infimum, an upper bound only requires one convenient admissible choice of $\beta$ for each $L$. We do not claim that the following choice is the finite-order optimizer. To make the geometric contribution negligible on the $L^{-2}$ scale, take
\begin{equation}\label{eq:a-log-scale}
  \beta_L=\frac{3\log L}{2\pi}.
\end{equation}
Then $q_L=L^{-3}$ and
\begin{equation}\label{eq:replica-negligible}
  L^2\frac{q_L}{(1-q_L)^2}\longrightarrow0.
\end{equation}
Thus, the independent period index contributes $o(L^{-2})$ to the variance. The price of this choice is that $\e^{2\pi\beta_L}=L^3$. Within the first period, any mass that wraps from just beyond $2\pi$ back toward zero is amplified by a factor $L^3$. The first-period construction below is designed precisely to retain the $L^{-2}$ concentration scale while suppressing this wrapped tail strongly enough to overcome that amplification.

\subsection{A concentrated circular law}
The quantity $4\sin^2(u/2)=|\e^{\ii u}-1|^2$ is the squared chordal distance from $u=0$ on the unit circle. We construct an auxiliary circular law whose expectation of this distance is of order $m^{-2}$. Notice that this circular law is not yet the random clock on $[0,\infty)$. It will instead be used to design the periodic envelope that is unfolded into ordinary time in the next subsection.

For an integer $m\geq1$, put
\[
  s_k=\sin\!\left(\frac{(k+1)\pi}{m+2}\right),\qquad
  S_m(u)=\sum_{k=0}^{m}s_k\e^{\ii ku},
\]
and let $U_m$ have density
\[
  h_m(u)=\frac{\abs{S_m(u)}^2}
                   {2\pi\sum_{k=0}^{m}s_k^2},
  \qquad -\pi\leq u<\pi.
\]
Parseval's identity shows that the denominator normalizes $h_m$ to integrate to one. The coefficients are real, so $h_m$ is even. Write
\[
  \ell_m=4\sin^2\!\left(\frac{\pi}{2(m+2)}\right).
\]
The sine coefficients are not ad hoc. With $s_{-1}=s_{m+1}=0$, the vector $(s_0,\ldots,s_m)^{\top}$ is the first eigenvector of the discrete Dirichlet Laplacian, the unscaled centered second-difference matrix for $-u''$ with homogeneous Dirichlet boundary conditions,
\[
  (D_m\bm s)_k=2s_k-s_{k-1}-s_{k+1},\qquad 0\leq k\leq m,
\]
whose eigenvalues are $4\sin^2(j\pi/(2(m+2)))$, $j=1,\ldots,m+1$ \citep[Section~3]{StrangMacNamara2014}, so $\ell_m$ is the smallest one. The following identity is, by Parseval, the corresponding Rayleigh quotient $\bm s^{\top}D_m\bm s/(\bm s^{\top}\bm s)$.

\begin{lemma}\label{lem:sine-moment}
We have
\begin{equation}\label{eq:circular-moment}
  \E\!\left[4\sin^2(U_m/2)\right]=\ell_m,
  \qquad m^2\ell_m\longrightarrow\pi^2.
\end{equation}
\end{lemma}

\begin{proof}
For any finite coefficient sequence $a_0,\ldots,a_d$, integration of the cross terms gives
\begin{equation}\label{eq:finite-parseval}
  \frac1{2\pi}\int_{-\pi}^{\pi}
    \left|\sum_{k=0}^{d}a_k\e^{\ii ku}\right|^2\dd u
  =\sum_{k=0}^{d}|a_k|^2.
\end{equation}
Indeed, the integral of $\e^{\ii(j-k)u}/(2\pi)$ is one when $j=k$ and zero otherwise. This is the finite-polynomial case of Parseval's identity. See \citet[Chapter~3, Theorem~1.3(ii)]{SteinShakarchi2003}. Applying \eqref{eq:finite-parseval} to $S_m$ and $(1-\e^{\ii u})S_m(u)$ gives
\[
  \E\!\left[4\sin^2(U_m/2)\right]
  =\frac{s_0^2+\sum_{k=1}^{m}(s_k-s_{k-1})^2+s_m^2}
         {\sum_{k=0}^{m}s_k^2}.
\]
Set $s_{-1}=s_{m+1}=0$. The sine addition formula gives
\[
  2s_k-s_{k-1}-s_{k+1}=\ell_m s_k,
  \qquad 0\leq k\leq m.
\]
Multiplying by $s_k$ and summing shows that the numerator above is $\ell_m\sum_{k=0}^{m}s_k^2$. The limit follows from $\sin x/x\to1$ as $x\to0$.
\end{proof}

\subsection{Quadratic concentration}
We now transfer the circular concentration of $U_m$ to ordinary-time concentration of the first-period variable, while resolving the trade-off created by the growing damping $\beta_L$. The circular peak alone has only polynomial control of fixed-distance tails. After unfolding near $2\pi$, such a tail would be amplified by the factor $L^3$. We therefore multiply the sine-polynomial envelope by a cosine power. The proof will split the second moment about the target point into an unwrapped part with asymptotic upper bound $(\pi^2+o(1))/L^2$, and a wrapped part, which the cosine power makes negligible.

Fix the within-period peak location $\vartheta\in(\pi,2\pi)$ and write $\delta=2\pi-\vartheta$. The restriction $\vartheta>\pi$ makes wrapping one-sided when the circular coordinate $u\in[-\pi,\pi)$ is unfolded onto $[0,2\pi)$. Points can cross the right endpoint $2\pi$ but not the left endpoint $0$. The eventual SCV bound is proportional to $1/\vartheta^2$, so we shall later take $\vartheta$ as close to $2\pi$ as possible. For $L\geq4$, set
\[
  r_L=\lfloor\sqrt L\rfloor,\qquad m_L=L-r_L,
\]
and define
\begin{equation}\label{eq:sine-cosine-envelope}
  T_{L,\vartheta}(z)
  =\abs{S_{m_L}(z-\vartheta)}^2
    \cos^{2r_L}\!\left(\frac{z-\vartheta}{2}\right).
\end{equation}
Let $Z_{L,\vartheta}$ have density proportional to $\e^{-\beta_Lz}T_{L,\vartheta}(z)$ on $[0,\infty)$, with $\beta_L$ given by \eqref{eq:a-log-scale}. The two factors in \eqref{eq:sine-cosine-envelope} have trigonometric degrees $m_L$ and $r_L$, so $T_{L,\vartheta}\in\Tcal_L^+$ and the distribution has an ME representation of order at most $N=2L+1$. The sine-weighted coefficient sequence supplies the constant in \eqref{eq:circular-moment}. The extra cosine power suppresses mass away from $\vartheta$, while using only $r_L=o(L)$ of the available degree. The choice $r_L=\lfloor\sqrt L\rfloor$ has three useful properties. First, $r_L\to\infty$ gives increasing tail suppression. Second, $r_L=o(L)$ leaves $m_L/L\to1$ so the sine polynomial retains the leading $L^{-2}$ constant. Third, $\beta_L^2/r_L\to0$ ensures that the exponential tilt does not alter that leading constant.

\begin{theorem}[Quadratic harmonic construction]\label{thm:quadratic-construction}
For every fixed $\vartheta\in(\pi,2\pi)$, the preceding construction satisfies
\begin{equation}\label{eq:construction-moments}
  \E[Z_{L,\vartheta}]\longrightarrow\vartheta,
  \qquad
  \limsup_{L\to\infty}L^2\Var(Z_{L,\vartheta})\leq\pi^2.
\end{equation}
Consequently,
\begin{equation}\label{eq:construction-bound}
  \limsup_{L\to\infty}(2L+1)^2\SCV(Z_{L,\vartheta})
  \leq\frac{4\pi^2}{\vartheta^2}.
\end{equation}
\end{theorem}

\begin{proof}
Write $m=m_L$, $r=r_L$ and $\beta=\beta_L$, and let $\Theta_{L,\vartheta}$ denote the first-period component in Proposition~\ref{prop:period-decomposition}. We bound its second moment about $\vartheta$, which also bounds its variance.

Use $u\in[-\pi,\pi)$ as the representative of $z-\vartheta$ modulo $2\pi$, and put $w_L(u)=h_m(u)\cos^{2r}(u/2)$. For $u<\delta$ the position in the first period is $z=\vartheta+u$, and for $u\geq\delta$ it is $z=\vartheta+u-2\pi$. After cancelling the common factors, the normalizing constant is therefore
\[
  D_L=\int_{-\pi}^{\delta}\e^{-\beta u}w_L(u)\dd u
        +L^3\int_{\delta}^{\pi}\e^{-\beta u}w_L(u)\dd u.
\]
The factor $L^3=\e^{2\pi\beta}$ accounts for the part of the circular profile that wraps into the beginning of the first period. The normalizer $D_L$ is kept explicit because the following estimates first control unnormalized second-moment contributions. The lower bound $D_L\geq1-o(1)$ established below ensures that normalization does not increase the leading asymptotic upper constant.

We first bound $D_L$ from below, using the elementary inequality $1-(1-x)^r\leq rx$ for $0\leq x\leq1$, which gives
\[
  \E\!\left[1-\cos^{2r}(U_m/2)\right]\leq\frac{r\ell_m}{4}.
\]
Also, by \eqref{eq:circular-moment},
\[
  \Pr(|U_m|\geq\delta)
  \leq\frac{\ell_m}{4\sin^2(\delta/2)}.
\]
Since $w_L$ is even and $[-\delta,\delta]\subset[-\pi,\delta]$, pairing the contributions at $u$ and $-u$ replaces the exponential weights by $\cosh(\beta u)\geq1$, and we obtain
\begin{align}
  D_L
  &\geq\int_{-\delta}^{\delta}\cosh(\beta u)w_L(u)\dd u
   \geq\int_{-\delta}^{\delta}w_L(u)\dd u \notag\\
  &\geq1-\frac{r\ell_m}{4}
          -\frac{\ell_m}{4\sin^2(\delta/2)}
   =1-o(1).
  \label{eq:normalizer-bound}
\end{align}

Next, consider the second moment on the part that does not wrap. Here, the potentially increasing factor $\e^{\beta|u|}$ from the exponential tilt competes with the Gaussian-type decay supplied by the cosine power. The exponent $\beta|u|-(r-1)u^2/4$ below records exactly this competition. For $|u|\leq\pi$,
\begin{equation}\label{eq:elementary-cosine-bounds}
  u^2\cos^2(u/2)\leq4\sin^2(u/2),
  \qquad
  \cos(u/2)\leq\e^{-u^2/8}.
\end{equation}
The first inequality follows from $x\leq\tan x$ for $0\leq x<\pi/2$. Integrating that inequality gives $-\log\cos x\geq x^2/2$, and hence the second, with the endpoint values following by continuity. As $r>1$, completing the square yields
\begin{align*}
  u^2\e^{-\beta u}\cos^{2r}(u/2)
  &\leq4\sin^2(u/2)
      \exp\!\left(\beta|u|-\frac{r-1}{4}u^2\right)\\
  &\leq4\sin^2(u/2)\exp\!\left(\frac{\beta^2}{r-1}\right).
\end{align*}
Integration against $h_m$ thus bounds the unnormalized second-moment contribution from $u<\delta$ by $\ell_m\exp(\beta^2/(r-1))$.

On the wrapped part, $u\geq\delta>0$ and the displacement from $\vartheta$ is $u-2\pi$. Every wrapped point is therefore at least circular distance $\delta$ from the peak, and $\cos^{2r}(u/2)\leq\exp(-ru^2/4)\leq\exp(-r\delta^2/4)$. This exponential suppression is what defeats the $L^3$ amplification caused by wrapping. Therefore
\[
  L^3\int_{\delta}^{\pi}(u-2\pi)^2\e^{-\beta u}w_L(u)\dd u
  \leq4\pi^2L^3\e^{-r\delta^2/4}.
\]
Here, we used $(u-2\pi)^2\leq4\pi^2$, $\e^{-\beta u}\leq1$, \eqref{eq:elementary-cosine-bounds}, and $\int_{\delta}^{\pi}h_m\leq1$. Since $r$ is of order $\sqrt L$ and $\delta>0$ is fixed, $L^3\e^{-r\delta^2/4}=o(L^{-2})$. The exponential decay in $\sqrt L$ dominates every polynomial factor. Combining the two contributions gives
\[
  \E\!\left[(\Theta_{L,\vartheta}-\vartheta)^2\right]
  \leq\frac{\ell_m\exp(\beta^2/(r-1))
                  +4\pi^2L^3\e^{-r\delta^2/4}}{D_L}.
\]

Now $m/L\to1$, $L^2\ell_m\to\pi^2$, and $\beta^2/(r-1)\to0$. Because $\delta$ is fixed and positive, the wrapped bound is $o(L^{-2})$. Together with \eqref{eq:normalizer-bound}, this proves
\[
  \limsup_{L\to\infty}
  L^2\E\!\left[(\Theta_{L,\vartheta}-\vartheta)^2\right]\leq\pi^2.
\]
In particular, $\E[\Theta_{L,\vartheta}]\to\vartheta$ by Cauchy--Schwarz, and the same upper bound holds for $L^2\Var(\Theta_{L,\vartheta})$.

Finally, Proposition~\ref{prop:period-decomposition} adds an independent geometric number of periods with parameter $q_L=L^{-3}$. Its mean contribution vanishes and its variance contribution is $o(L^{-2})$ by \eqref{eq:replica-negligible}. This proves \eqref{eq:construction-moments}. Dividing the variance by the squared mean and using $(2L+1)^2/L^2\to4$ gives \eqref{eq:construction-bound}.
\end{proof}

For fixed $\vartheta$, the construction therefore has an absolute variance of order at most $\pi^2/L^2$ and an asymptotic mean $\vartheta$. Since the resulting SCV bound $4\pi^2/\vartheta^2$ is smallest as $\vartheta$ approaches $2\pi$, we let $\vartheta\uparrow2\pi$ in the corollary below.
\begin{corollary}\label{cor:limsup}
The harmonic optimum satisfies
\begin{equation}\label{eq:limsup}
  \limsup_{L\to\infty}(2L+1)^2\vH_L\leq1.
\end{equation}
The same bound is attainable within the cosine-square product family.
\end{corollary}

\begin{proof}
For every fixed $\vartheta\in(\pi,2\pi)$, $\vH_L\leq\SCV(Z_{L,\vartheta})$, so applying Theorem~\ref{thm:quadratic-construction} and then letting $\vartheta\uparrow2\pi$, an order of limits that requires no estimate uniform in $\vartheta$, gives the bound. The last statement follows from Corollary~\ref{cor:cosine-equivalence}.
\end{proof}

Here, $\vartheta$ is a location within the rescaled harmonic period, not the physical deterministic time being approximated. Since the SCV is invariant under positive rescaling, once a concentrated clock has been constructed it may be normalized to mean one and then multiplied by any desired time $t_0$.

\section{Numerical findings and implementation}\label{sec:numerics}

This section compares numerical values obtained from the reduction of Theorem~\ref{thm:two-parameter} with previously reported values, where previous computations replaced the full offset search by low-dimensional heuristic parametrizations \citep{HorvathHorvathTelek2020,AlmousaTelek2021}.

For any admissible recovered kernel, its exact SCV is an upper bound on $\vH_L$, whether or not the outer search has reached a global minimum. This does not by itself certify a floating-point approximation to that SCV or to the associated eigenvalue. The practical question at high order is therefore the floating-point accuracy with which $\lambda_L(\beta,\theta)$ itself is evaluated: the Toeplitz matrices in \eqref{eq:moment-matrix} can combine comparable-magnitude terms into a much smaller generalized eigenvalue, risking cancellation, and a spurious downward perturbation could make some parameter values appear better than they actually are. A Cholesky change of basis, so that $M_0(\beta)$ becomes the identity, reduces the pencil to a standard Hermitian eigenvalue problem \citep[Table~2.13, case~1]{AndersonEtAl1999}. This reduction alone does not certify the accuracy of a small computed eigenvalue. A recovered kernel can be checked independently by recomputing its moments at higher precision.

For comparison, we use the three-parameter heuristic benchmark of \citet{AlmousaTelek2021}, restricted to their reported values since the six-parameter heuristic they also develop is not evaluated here. Write $\widehat{\vH_L}$ for the tabulated approximation to $\vH_L$, obtained by transforming the numerically computed smallest generalized eigenvalue $\lambda_L(\beta,\theta)$ through \eqref{eq:v-from-rho}, and $v_L^{\mathrm{prev}}$ for the published three-parameter heuristic SCV. Table~\ref{tab:comparison} lists both at each order, together with the percentage reduction $(1-\widehat{\vH_L}/v_L^{\mathrm{prev}})\times100\%$. This reduction is about $46\%$ at the two smallest orders and ranges from about $65\%$ to $79\%$ at the larger orders considered, though not monotonically.

\begin{table}[H]
\centering
\caption{
{Comparison between the three-parameter heuristic values of \citet{AlmousaTelek2021} ($v_L^{\mathrm{prev}}$) and the numerical approximation to the harmonic infimum ($\widehat {\vH_L}$). The percentage improvement of the latter over the former is shown in the last column. The scenarios considered are those of Table~2 of \citet{AlmousaTelek2021}.}
}
\label{tab:comparison}
\begin{tabular}{rrrrr}
\toprule
$L$ & $N$ & {$v_L^{\mathrm{prev}}$} & { $\widehat {\vH_L}$} & {$(1-\widehat {\vH_L}/v_L^{\mathrm{prev}})\times100\%$} \\
\midrule
400  & 801  & $3.5945\times10^{-6}$  & $1.9403\times10^{-6}$ & $46.0\%$ \\
800  & 1601 & $8.53737\times10^{-7}$ & $4.5915\times10^{-7}$ & $46.2\%$ \\
1200 & 2401 & $3.69091\times10^{-7}$ & $1.1003\times10^{-7}$ & $70.2\%$ \\
1500 & 3001 & $2.32831\times10^{-7}$ & $8.1685\times10^{-8}$ & $64.9\%$ \\
2000 & 4001 & $1.28656\times10^{-7}$ & $3.2833\times10^{-8}$ & $74.5\%$ \\
2500 & 5001 & $8.12596\times10^{-8}$ & $1.7038\times10^{-8}$ & $79.0\%$ \\
\bottomrule
\end{tabular}
\end{table}

For the asymptotic discussion it is also useful to look at the scaled values $N^2{\widehat{\vH_L}}$. Figure~\ref{fig:scaled-scv} shows the selected numerical sequence used in the computations. After a small low-order maximum, the scaled values decrease steadily over the displayed range and reach about $1.18$ at $N=1601$. The higher-order entries in Table~\ref{tab:comparison} give scaled outputs of approximately $0.634$, $0.736$, $0.526$ and $0.426$ at $L=1200, 1500, 2000, 2500$, respectively. The horizontal level one is an asymptotic upper-bound reference, not a finite-order bound or a proved limiting value.

\begin{figure}[H]
\centering
\includegraphics[width=0.78\textwidth]{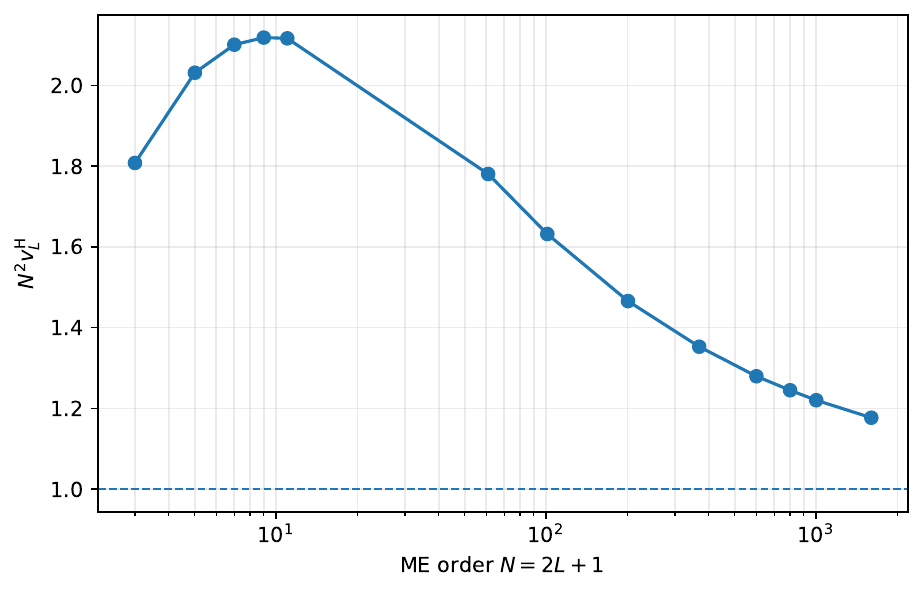}
\caption{Selected numerical {estimates} of $N^2{\widehat{\vH_L}}$, with
$N=2L+1$. The horizontal reference level is the asymptotic upper constant
one from {Corollary~\ref{cor:limsup}}.}
\label{fig:scaled-scv}
\end{figure}

\section{Discussion}\label{sec:discussion}

The harmonic SCV infimum admits an exact characterization through a two-dimensional nonlinear search, with the coefficient optimization handled by a generalized eigenvalue problem. Independently, the explicit construction proves $\limsup_{L\to\infty}(2L+1)^2\vH_L\leq1$.

It is useful to compare this with the polynomial-square repeated-pole class studied in \citet{BattagliolaPeralta2026PolynomialSquare}. Both constructions improve the Erlang SCV $1/N$ to a quadratic scale, though the mechanisms differ. In this class, the kernel is an exponential factor multiplied by the square of a degree-$m$ ordinary polynomial. Let $v_m$ denote the optimal SCV in this class, whose representation budget is $2m+1$. The polynomial-square problem is solved by Laguerre zero geometry and has the sharp asymptotic law
\[
  (2m+1)^2v_m\longrightarrow\pi^2,
\]
while the harmonic construction proved here gives
\[
  \limsup_{L\to\infty}(2L+1)^2\vH_L\leq1.
\]
The two classes have different pole architectures, so this is a comparison rather than an inclusion statement. It nevertheless shows quantitatively how allowing a harmonic array of complex poles can produce substantially more concentrated random clocks than the single repeated-real-pole polynomial-square architecture.

The proved asymptotic result is the upper bound \eqref{eq:limsup}. As with any finite-order computation, the numerical results in Section~\ref{sec:numerics} cannot by themselves certify whether this asymptotic constant is sharp. We leave that question for future work.

\end{document}